\documentclass[preprint, 11pt]{amsart}
\usepackage{amsfonts,amssymb,amsmath,amsthm,amscd,mathtools, multicol,tikz, tikz-cd,caption,enumerate,mathrsfs,geometry,tgpagella}
\usetikzlibrary{arrows,chains,matrix,positioning,scopes}
\usepackage{inputenc}
\usepackage{todonotes}
\usepackage[pagebackref=true, colorlinks, linkcolor=blue, citecolor=red]{hyperref}
\usepackage{fullpage}
\usepackage[all]{xy}

\makeindex
\usepackage[linesnumbered,ruled,vlined,hanginginout]{algorithm2e}
\SetKwInput{KwInput}{Input}
\SetKwInput{KwOutput}{Output}

\newtheorem{theorem}{Theorem}[section]
\newtheorem{proposition}[theorem]{Proposition}
\newtheorem{definition}[theorem]{Definition}
\newtheorem*{theorem*}{Theorem}
\newtheorem{corollary}[theorem]{Corollary}

\theoremstyle{remark}
\newtheorem{remark}[theorem]{Remark}
\newtheorem{example}[theorem]{Example}

\numberwithin{equation}{subsection}

\title{Iterated and Generalized Iterated Integrals}

\author{Chitrarekha Sahu}
\address{Indian Institute of Science Education and Research Mohali, Punjab 140306, India}
\email{chitrarekhasahu97@gmail.com}

\author{Matthias Sei\ss}
\address{Institut für Mathematik, Universit\"at Kassel, 34109 Kassel, Germany}
\email{mseiss@mathematik.uni-kassel.de}

\author{V. Ravi Srinivasan}
\address{Indian Institute of Science Education and Research Mohali, Punjab 140306, India}
\email{ravisri@iisermohali.ac.in}

\thanks{The first named author would like to acknowledge the support of DST-FIST grant SR/FST/MS-I/2019/46(C)}
\date{\today}
\begin{document}
	\begin{abstract}
		For a differential field $F$ having an algebraically closed field of constants, we analyze the structure of Picard-Vessiot extensions of $F$ whose differential Galois groups are unipotent algebraic groups and apply these results to study stability problems in integration in finite terms and the inverse problem in differential Galois theory for unipotent algebraic groups.
		\end{abstract}
	\maketitle
	
	\section{Introduction}
		
Let $F$ be a differential field of characteristic zero with an algebraically closed field of constants $C$ and denote by $F^*$ the nonzero elements of $F$. For an element $f\in F$ we denote its first derivative by $f'$ and its $n$-th derivative by $f^{(n)}$. Let $F[\partial]$ be the ring of linear differential operators over $F$ with the usual rule $\partial f = f \partial +f'$. For $\mathbf f:=(f_1,\dots, f_l)\in   (F^{*})^{l}$ define  
$$\mathcal L_{\mathbf f}:=f_1\partial f_2\partial\dots f_{l-1}\partial f_l\partial\in F[\partial].$$
	  Let $E$ be a differential field extension of $F$ having $C$ as its field of constants. An element $\eta\in E$ is said to be an \emph{integral} of an element $f\in F$ if $\eta'=f$. For convenience, we shall write $\int f$ to denote an integral of $f.$ An element $\eta\in E$ is called a \emph{generalized iterated integral} (respectively, an \emph{iterated integral}\footnote{Our definition of an iterated integral is equivalent to the following one presented in \cite[Section 7.2]{Roq-Sin}: An element $\eta\in E$ is called an \emph{iterated integral} of $F$, if $\eta^{(l)}= g \in F.$  To see this, observe that if $g=0$, then for $\mathbf f:=(1,1,\dots,1)\in (F^{*})^{l}$ we have $\mathcal L_\mathbf f(\eta)=0$ and if $ g \neq 0$, then for $\mathbf f:=(1,1/g, 1,\dots,1)\in (F^{*})^{l+1}$ we have $\mathcal L_\mathbf f(\eta)=0.$}) 
      of $F$ if there is a positive integer $l$ and an element $\mathbf f:=(f_1,f_2,\dots, f_{l})\in (F^{*})^{l}$ (respectively, $\mathbf f:=(1, f_2 ,1,\dots,1)\in (F^{*})^{l}$) such that $\mathcal L_{\mathbf f}(\eta)=0.$

	In this article, we prove the following theorems.
\begin{theorem} \label{intro-main} Let $E$ be a Picard-Vessiot extension of $F$ for a monic operator $\mathcal L\in F[\partial]$ of order $n$ and suppose that $F$ contains an element $x$ with $x'=1.$ 
	\begin{enumerate}[(i)]
	   \item \label{noncommutative-giint}  The following statements are equivalent. 
       \begin{enumerate} 
       \item The differential Galois group of $E$ over $F$ is a unipotent linear algebraic group. \item The Picard-Vessiot ring of $E$ is the set of all generalized iterated integrals in $E$ of $F.$ 
       \item The field $E$ is generated as a differential field over $F$ by a generalized iterated integral of $F.$ In particular, there exists $\mathbf g\in (F^*)^{l}$ such that $E$ is a Picard-Vessiot extension for the differential operator $\mathcal L_\mathbf g.$

\item The operator $\mathcal L$ can be factored as $$\mathcal L= \mathcal L_\mathbf ff_{n+1},$$ for some $f_{n+1}\in F^*$ and $\mathbf f=(f_1,f_2,\dots,f_n)\in (F^*)^n,$ in which case, there is a $C-$basis $v_1,v_2,\dots,v_n\in E$ for the solution space of $\mathcal L(y)=0$ such that $$v_1=\frac{1}{f_{n+1}}, v_2=\frac{1}{f_{n+1}}\int\frac{1}{f_n}, v_3=\frac{1}{f_{n+1}}\left(\int\frac{1}{f_n}\left(\int\frac{1}{f_{n-1}}\right)\right),\cdots,$$ $$v_n=\frac{1}{f_{n+1}}\left(\int\frac{1}{f_n}\left(\int\frac{1}{f_{n-1}}\cdots\frac{1}{f_3}\left(\int\frac{1}{f_2}\right)\cdots\right)\right).$$
 \\

\end{enumerate}

		\item \label{commutative-itint} The following statements are equivalent. 
        \begin{enumerate}
        \item The field $E$ is generated as a differential field over $F$ by an iterated integral of $F.$ \item The differential Galois group of $E$ over $F$ is isomorphic to a commutative unipotent algebraic group. \end{enumerate}
		\end{enumerate}

		\end{theorem}

\begin{theorem}[Generalization of Liouville's Theorem]  \label{genofliouvilleintroduction} 
Let $g$ be an element of $F.$
\begin{enumerate}[(i)] 
\item Suppose that there exists an element  $x\in F$ with $x'=1.$  Then there is a positive integer $n$ and an element $\eta$ in some elementary extension\footnote{A differential field extension $E$ of $F$ is said to be an elementary extension of $F$, if $E$ has $C$ as its field of constants and there exist nonzero elements $\eta_1,\dots, \eta_n\in E$ such that $E=F(\eta_1,\dots,\eta_n)$ and  that either there is  a nonzero $\alpha_i\in F(\eta_1,\dots,\eta_{i-1})$ such that  $\eta'_i=\alpha'_i/\alpha_i$ or there is an $\alpha_i\in F(\eta_1,\dots,\eta_{i-1})$ with $\eta'_i/\eta_i=\alpha'_i$ or $\eta_i$ is algebraic over $F(\eta_1,\dots,\eta_{i-1}).$} 
of $F$ such that $\eta^{(n)}=g$, if and only if there exist $f, u_1,u_2,...,u_m\in F$ and $f_1,f_2,...,f_m \in C[x]$ with $\deg_x(f_i)\leq n-1$ for all $1\leq i\leq m$ such that  
$$g =f^{(n)}+\sum_{i=1}^{m}\sum_{j=1}^{n} f_i^{(n-j)}\left(\frac{u_i'}{u_i}\right)^{(j-1)}.$$

\item Suppose that $z'\neq 1$ for any $z\in F.$ Then there is a positive integer $n$ and an element $\eta$ in some elementary extension of $F$ such that $\eta^{(n)}=g$, if and only if 
there exist $f, u_1,u_2,...,u_m\in F$ and constants $c, c_1,c_2,...,c_m$ such that 
$$g=c+f^{(n)}+\sum_{i=1}^{m}c_i\left(\frac{u_i'}{u_i}\right)^{(n-1)}.$$
\end{enumerate}
\end{theorem}    
	
In \cite[Section 7.2]{Roq-Sin}, it was shown that if a Picard-Vessiot extension $E$ of $F$ is generated as a differential field over $F$ by an iterated integral of $F$, then the differential Galois group of $E$ over $F$ is a unipotent algebraic group. Thus, Theorem~\ref{intro-main} (\ref{commutative-itint}) is a refinement of their result.

Using Theorem~\ref{intro-main}, we study stability problems in symbolic integration (see \cite{chen2022stability}) and prove Theorem~\ref{genofliouvilleintroduction}, which is a generalization of Liouville's theorem (cf.~\cite{rosenlicht1968liouville}) in integration in finite terms. 

Given a unipotent linear algebraic group $U$, Kovacic presented in \cite{KovacicInvProb} a method to construct a linear differential equation having $U$ as its differential Galois group.
His equation is a matrix differential equation defined by an element of the Lie algebra of $U$. Using his result, we produce a method of constructing $\mathbf f \in (F^*)^{n}$ such that the differential Galois group of $\mathcal L_\mathbf{f}$ is isomorphic to $U.$

	\section{picard-vessiot extensions with unipotent differential galois groups.}  
    
    Throughout this section, $E$ denotes a differential field extension of $F$  having $C$ as its field of constants.
    The group of all differential $F-$automorphisms of $E$, that is, the group of all $F$-automorphisms of $E$ commuting with the derivation, is called the \emph{differential Galois group} and is denoted by 
	$$\mathscr G(E|F):= \{\sigma \in \mathrm{Aut}(E|F)\ | \ \sigma(x')=\sigma(x)' \ \text{for all} \ x\in E\}.$$
	Let 
    $$T(E|F):=\{x\in E\ | \ \mathcal L (x)=0 \ \text{for some} \ \mathcal{L} \in F[\partial]\}.$$ 
    It is well known that $T(E|F)$ is a differential $F$-algebra \cite[Corollary 1.38]{MvdP03}. 
    For any $\sigma\in \mathscr G(E|F)$ we observe that $\mathcal L(\sigma(x))=\sigma(\mathcal L(x))$ for all $x\in E$ and $\mathcal L\in F[\partial].$ 
    Thus, $\mathscr G(E|F)$ is also a group of differential $F-$automorphisms of $T(E|F).$  
    If $E$ is a Picard-Vessiot extension of $F,$ then $T(E|F)$ is called a \emph{Picard-Vessiot ring} of $E$ and its field of fractions is $$\mathrm{Frac}(T(E|F))=E.$$

 In the next two propositions, for convenience and easy reference,  we shall state several fundamental results from the differential Galois theory.
 
	\begin{proposition} \label{prop:fundamentalresults1} $\,$
    \begin{enumerate}[(i)]
		\item  \label{item1:fundamentalresults1}
        Let $F(x)$ be the rational function field over $F$ in the variable $x$ and extend the derivation of $F$ to $F(x)$ by defining $x'=g$ for some $g \in F.$ Then the field of constants of $F(x)$ is the same as the field of constants of $F$, if and only if $f'\neq g$ for all $f \in F$ (cf.~\cite[Theorem 1.3]{vrs-20}). 
		\item Let $E=F(\zeta_1,\dots, \zeta_m),$ where $\zeta'_1\in F$ and $\zeta'_i \in F(\zeta_1,\dots,\zeta_{i-1})$  for each $2\leq i\leq m$. Then, there are $F-$algebraically independent elements $\eta_1,\dots, \eta_n $ among $\zeta_1,\dots, \zeta_m$ such that $E=F(\eta_1,\dots, \eta_n)$ and that $\eta'_1\in F$ and $\eta'_i\in F(\eta_1,\dots, \eta_{i-1})$  for each $2\leq i\leq n$  (cf.~\cite[Theorem 2.1]{vrs-10}).
	\end{enumerate}
		\end{proposition}

\begin{proposition}	\label{Group-Extension} Let $E$ be a Picard-Vessiot extension of $F.$
	\begin{enumerate}[(i)]
	\item \label{PV-IAE} 
    The differential Galois group $\mathscr G(E|F)$ is isomorphic to a unipotent algebraic group if and only if  
    $E=F(\eta_1,\dots, \eta_n)$, where $\eta'_1\in F$ and $\eta'_i\in F(\eta_1,\dots,\eta_{i-1})$ for each $2\leq i\leq n;$
    in which case $\eta_1,\dots,\eta_n$ can also be chosen from $T(E|F)$ so that $T(E|F)=F[\eta_1,\dots,\eta_n]$ and $\eta'_1\in F$ and $\eta'_i\in F[\eta_1,\dots,\eta_{i-1}]$ for each $2\leq i\leq n.$ 
	\item \label{PV-AE} 
    The differential Galois group $\mathscr G(E|F)$ is isomorphic to a commutative unipotent algebraic group if and only if $T(E|F)=F[\eta_1,\eta_2,\dots,\eta_n]$ with $\eta'_i \in F$ for each $1\leq i\leq n.$ 
\end{enumerate}
\end{proposition}
\begin{proof} (\ref{PV-IAE}) and (\ref{PV-AE}) follow from \cite[Propositions 6.7, 6.8]{mag-book} and \cite[Theorem 4.4]{KRS-24}). 
\end{proof}

Using these results, we now prove Theorem \ref{intro-main} (\ref{commutative-itint}), which also contains a refinement of \cite[Lemma 7.5]{Roq-Sin}.	
	\begin{theorem} \label{antiderivative-PV-commutative} 
    Assume that $F$ contains an element $x$ with $x'=1.$ Then $E$ is generated as a differential field by an iterated integral of $F$, if and only if $E$ is a Picard-Vessiot extension of $F$ with $\mathscr G(E|F)$ isomorphic to a commutative unipotent linear algebraic group. 
	\end{theorem}
	\begin{proof}
		Let $E=F\langle \eta\rangle,$ where $\eta$ is an iterated integral of $F$ with $\eta^{(n)}=:f\in F.$ 
        First, we observe that if $\eta \in F$, then $E=F$ and consequently the differential Galois group is trivial. Therefore, we may assume that $\eta$ is not an element of $F$. In particular, this implies that $f\neq 0$, since otherwise $\eta \in C[x]$. 
        Let
        $$V:={\rm{span}}_C\{\eta, 1, x, x^2,\dots, x^{n-1}\} \subset E$$
        be the $C$-vector space spanned by $\eta, 1, x, x^2,\dots, x^{n-1}$. Then $\mathrm{dim}(V)=n+1$ and $V$  is  the solution space of the linear homogeneous differential equation 
        \begin{equation}\label{diffeqn-ItInt}Y^{(n+1)}-\frac{f'}{f}Y^{(n)}=0 \, .
        \end{equation}
        Since $E=F\langle \eta\rangle=F\langle V\rangle$ and $E$ has by assumption $C$ as its field of constants, $E$ is a  Picard-Vessiot extension of $F.$ 
		
	Observe that $E=F\langle \eta \rangle=F(\eta^{(n-1)}, \eta^{(n-2)},\dots, \eta', \eta).$  Thus,  as noted in Proposition \ref{Group-Extension} (\ref{PV-IAE}), the differential Galois group $\mathscr G(E|F)$ is a unipotent algebraic group. Now we shall prove that $\mathscr G(E|F)$ is a commutative group. 
	
	Since $E=F\langle \eta\rangle,$ the Galois group $\mathscr G(E|F)$ becomes commutative at once if we show that $$\sigma(\tau (\eta))=\tau(\sigma(\eta)),\quad \text{for all } \sigma, \tau\in \mathscr G(E|F).$$
		
         Observe that $V$ is stabilized by the differential Galois group and in particular, for any $\sigma\in \mathscr G(E/F),$ we have
         $$\sigma(\eta)=c_\sigma \eta+f_\sigma,$$ 
         where $c_{\sigma}$ is a constant and $f_\sigma \in F$ is a $C-$linear combination of $1, x,\dots, x^{n-1}.$ 
         Since $\eta \in E \setminus F$ and $\sigma$ is an $F-$automorphism, we conclude that $c_{\sigma}\neq 0$. Since $\mathscr G(E/F)$ acts rationally on $V,$ the map 
         \[
         \phi:\mathscr G(E/F) \to C^*, \quad \sigma \mapsto c_{\sigma}
         \]
         is a morphism of varieties. We are going to show that $\phi$ is a morphism of algebraic groups. Since $f_\sigma \in F$, we have    $\tau(f_\sigma)=f_\sigma$ for all $\sigma, \tau\in \mathscr G(E/F)$ 
		 and so we obtain 
        \[ \sigma(\tau(\eta))=\sigma(c_{\tau} \eta + f_{\tau} ) =c_{\sigma}c_{\tau}\eta+ f_{\tau}+c_{\tau}f_{\sigma} . \]
        Since we also have $\sigma(\tau(\eta))=c_{\sigma\tau}\eta+f_{\sigma\tau}$, 
        we find 
        \begin{eqnarray}
			c_{\sigma}c_{\tau}&=&c_{\sigma\tau} , \label{eqn:algebraicgroupstruction} \\
			f_{\tau}+c_{\tau}f_{\sigma}&=& f_{\sigma\tau} . \nonumber
		\end{eqnarray}
		It follows now from \eqref{eqn:algebraicgroupstruction} that the map 
        $\phi: \mathscr G(E/F)\to \mathrm{G}_m(C)$ is a morphism of algebraic groups. 
        But, since $\mathscr G(E/F)$ is unipotent, $\phi$ must be trivial, that is $c_\sigma=1$ for all $\sigma\in \mathscr G(E/F).$ 
        We conclude that $$\sigma(\tau(\eta))=\eta+f_{\tau}+f_{\sigma}=\eta+f_{\sigma}+f_{\tau}=\tau(\sigma(\eta)$$ 
        for any $\sigma, \tau\in \mathscr{G}(E|F)$, i.e.~$\mathscr G(E|F)$ is a commutative unipotent group. 
		
		Conversely, let $\mathscr G(E|F)$ be a commutative unipotent algebraic group. Then by Proposition \ref{Group-Extension} (\ref{PV-AE}) we have $E=F(\eta_1,\dots,\eta_n)$ with $\eta'_i=f_i\in F$ for each $1\leq i\leq n.$ 
        Let $$\eta:=\eta_1+x\eta_2+\dots+x^{n-1}\eta_n.$$

		One can show now by induction on $1\leq m\leq n-1$ that there are nonzero numbers $c_{1,m},\dots, c_{n-m,m} \in \mathbb{N}$ and an element $g_m\in F$ such that
        \begin{equation*} 
		\eta^{(m)}=g_m+c_{1,m}\eta_{m+1}+c_{2,m}x\eta_{m+2}+\cdots+c_{n-m,m}x^{n-m-1}\eta_n .
        \end{equation*}
        For $m=n-1$ this yields $\eta^{(n-1)}=g_{n-1} + c_{1,n-1}\eta_{n}$ and so  $\eta^{(n)} \in F$. 
        Moreover, one can use the above inductive result to prove that for each $1\leq m\leq n$ we have   
        $$E=F(\eta_1,\dots, \eta_n)=F(\eta_n,\dots, \eta_{m+1}, \eta^{(m-1)},\dots, \eta',\eta).$$
        In particular, for $m=n$ we obtain $ E=F(\eta^{(n-1)}, \eta^{(n-2)},\dots, \eta',\eta)=F\langle \eta\rangle$.	
        \end{proof}
	
	\begin{proposition} \label{itofF(x)}
    Suppose that there is no element $f \in F$ with $f'=1$. Let $F(x)$ be the rational function field over $F$ in the variable $x$ and extend the derivation of $F$ to $F(x)$ by defining $x'=1$. 
    Let $\eta \in F(x)$ be an iterated integral of $F$ with $\eta^{(r)}\in F.$ Then $\eta=f+P$ for some $f \in F$ and $P\in C[x]$ with $\deg(P)\leq r.$
	\end{proposition}
	
	\begin{proof} 
    The assumption on $F$ together with Proposition~\ref{prop:fundamentalresults1} \eqref{item1:fundamentalresults1} imply that the field of constants of $F(x)$ is $C$. Then the field $F(x)$ is a Picard-Vessiot extension of $F$ for $y''=0$ and it can be seen that $T(F(x)|F)=F[x]$.
   Since $\eta$ is an iterated integral, we must have $\eta \in F[x]$, that is, 
    $$\eta=f_0+f_1x+\cdots+f_nx^n,$$
	with $f_0,\dots,f_n\in F$ and $f_n\neq 0.$  We claim that $f_1,\dots, f_n$ are constants. 
    Suppose on the contrary that there is an index $i$ such that $f_i$ is not a constant. 
    Let $l$ be the largest index $1\leq l\leq n$ such that $f_l$ is not a constant. Note that any element of $C[x]$ is an iterated integral and a sum of iterated integrals is again an iterated integral. Since $-f_nx^n-\dots-f_{l+1}x^{l+1} \in C[x]$ it follows that   
    $$\tilde{\eta}:= f_0+f_1x+\cdots+f_lx^l = \eta -f_nx^n-\dots-f_{l+1}x^{l+1} $$
    is an iterated integral of $F$ with $f_l$ not a constant. Let $\tilde{r}$ be the smallest integer such that $\tilde{\eta}^{(\tilde{r})} \in F.$ Then, we have 
    $$ \tilde{\eta}^{(\tilde{r})}=f^{(\tilde{r})}_lx^l+ \left(lf^{(\tilde{r}-1)}_l+f^{(\tilde{r})}_{l-1}\right)x^{l-1}+\dots \in F.$$	
	Therefore, $f^{(\tilde{r})}_l=0$ and since $f_l$ is not a constant, we have that $\tilde{r}\geq 2.$
    Thus, there must exist an integer $s$ with $2\leq s\leq \tilde{r}$ such that $f^{(s)}_l=0$ and $f^{(s-1)}_l\neq 0$. For this particular $s$ we have that 
    $$\left( f^{(s-2)}_l/f^{(s-1)}_l\right)'=1 .$$
    This contradicts our assumption that $F$ has no elements whose derivative is equal to $1.$ Therefore, the elements $f_1, \dots, f_n$ must be constants. Thus, we have shown that $\eta=f_0+P,$ where $P\in C[x].$ Since $\eta^{(r)}\in F$ and $x$ is transcendental over $F,$ it readily follows that $\deg_x(P)\leq r.$\end{proof}

	Before we prove Theorem~\ref{intro-main} \eqref{noncommutative-giint}, we present the following remarks.
		\begin{remark}\label{integralsofiteratedintegrals}
		Let $\eta\in E$ and  $\mathbf g=(g_1,g_2,\dots, g_l)\in (F^*)^{l}.$ If $\mathcal{L}_{\mathbf g}(\eta) \in E$   is a generalized iterated integral, then so is $\eta.$ 
        Indeed, since $\mathcal{L}_{\mathbf g}(\eta)$ is a generalized iterated integral, there exists $\mathbf h=(h_1,h_2,\dots, h_m) \in (F^*)^{m}$ such that $\mathcal L_{\mathbf h}(\mathcal{L}_{\mathbf g}(\eta))=0$ and so for $$\mathbf f:=(h_1,h_2,\dots,h_m,g_1,\dots,g_l) \in (F^*)^{m+l}$$ we have $$\mathcal L_{\mathbf f}(\eta)=\mathcal L_{\mathbf h}(\mathcal{L}_{\mathbf g}(\eta))=0.$$
	\end{remark} 
	
	\begin{remark}\label{gii-iae}
		Let $\eta\in E$ be a generalized iterated integral of $F$ with $\mathcal L_\mathbf{f}(\eta)=0$ for $\mathbf f:=(f_1 , f_2,\dots, f_l)\in (F^*)^l.$ 
        For $1 \leq i\leq l$ we define $\mathbf{f}_i:=(f_i,\dots, f_l)\in (F^*)^{l-(i-1 )}$ and observe that  
        \[
        \mathcal L_{\mathbf f_1}(\eta)=\mathcal L_{\mathbf f}(\eta)=0 \quad \text{and}\quad (1/f_i)\mathcal L_{\mathbf f_i}(\eta)= \left(\mathcal L_{\mathbf f_{i+1}}(\eta)\right)'.
        \] 
	Thus, in particular, we have $(\mathcal L_{\mathbf f_2}(\eta))'=0$ and   $(1/f_l)\mathcal L_{\mathbf f_l}(\eta)=\eta'.$  
        Now for $1\leq i\leq l,$ let $\eta_i:=\mathcal L_{\mathbf f_i}(\eta)$ and $\eta_{l+1}:=\eta$. Observe that 
        \[
        \eta_1=0, \quad \eta_2\in C, \quad \eta'_3=(1/f_2)\eta_2\in F, \quad  \eta'_i\in F[\eta_3,\dots, \eta_{i-1}] \ \text{for} \ 3<i\leq l+1
        \]
        and that 
        $$F\langle \eta\rangle=F\left(\eta_3,\dots, \eta_{l+1} \right).$$
	\end{remark}

	\begin{theorem}\label{Uni-PVring}
		For a Picard-Vessiot extension $E$ of $F,$ the following statements are equivalent.  
		\begin{enumerate}[(i)]
		\item \label{diffeq-iae} There is an $\mathbf f\in (F^*)^l$ such that $E$ is a Picard-Vessiot extension of $F$ for $\mathcal L_{\mathbf f}\in F[\partial].$
		\item \label{group-iae}The differential Galois group $\mathscr G(E|F)$ is a unipotent linear algebraic group 
		\item \label{pvring-iae}The Picard-Vessiot ring $T(E|F)$ of $E$ is the set of all generalized iterated integrals in $E$ of $F.$ 
		\end{enumerate}
	\end{theorem}

\begin{proof} 
	\eqref{diffeq-iae} $\Rightarrow$ \eqref{group-iae}: 
    Let  $\mathbf f\in (F^*)^l$ such that $E$ is a Picard-Vessiot extension of $F$ for an operator $\mathcal L_\mathbf f$. Then there are $C-$linearly independent solutions $\eta_1,\dots,\eta_l \in E$ of $\mathcal L_\mathbf f(Y)=0$ such that 
    $E=F\langle \eta_1, \cdots, \eta_l\rangle.$ Note that for each $i$ the solution $\eta_i$ is a generalized iterated integral of $F.$ Since $E=F\langle \eta_1, \dots, \eta_n\rangle$ is a compositum of differential fields each of which is generated by a generalized iterated integral of $F$, it follows from Remark (\ref{gii-iae}) that there are $\zeta_1,\dots,\zeta_m \in E$ with $\zeta_1' \in F$ and $\zeta_i' \in F[\zeta_1,\dots,\zeta_{i-1}]$ for $2\leq i \leq m$ such that $E=F(\zeta_1,\dots,\zeta_m)$. 
    Now, from Proposition \ref{Group-Extension} \eqref{PV-IAE} we obtain that the differential Galois group $\mathscr G(E|F)$ is a unipotent algebraic group. 
	
	\eqref{group-iae} $\Rightarrow$ \eqref{pvring-iae}:  Let $\mathscr G(E|F)$ be a unipotent algebraic group. Then from Proposition \ref{Group-Extension} \eqref{PV-IAE}, we know that $T(E|F)=F[\eta_1,\dots, \eta_n],$ where $\eta'_1=:\alpha_1\in F$ and $\eta'_i=:\alpha_i\in F[\eta_1,\dots, \eta_{i-1}]$ for each $2\leq i \leq n$.   
    We shall use an induction to prove \eqref{pvring-iae}.
    Note that every element $f\in F$ is an iterated integral of $F$ showing the first step of the induction. Now suppose that there are  integers $1 \leq i\leq n$ and $l\geq 1$ such that  $F[\eta_1,\dots, 
	\eta_{i-1}]=T\left(F(\eta_1,\dots, 
	\eta_{i-1})\ |\ F\right)$ is the set of all generalized iterated integrals of $F$ in $F(\eta_1,\dots, 
	\eta_{i-1})$ and that every polynomial in $\eta_i$ over $ F[\eta_1,\dots, \eta_{i-1}]$ of degree $l-1$ is a generalized iterated integral of $F.$ Now, let  
	$$P=g_l\eta^l_i+g_{l-1}\eta^{l-1}_i+\dots+g_0\in F[\eta_1,\dots, \eta_{i-1}][\eta_i]$$ 
    be any polynomial of degree $l.$ 
	Then, since 
	$$P'=g'_l \eta^l_i+ (g'_{l-1}+lg_l \alpha_i)\eta^{l-1}_i+\dots+g_1\alpha_i+g'_0\in F[\eta_1,\dots, \eta_{i-1}][\eta_i],$$  for 
    $$\mathbf f:=\begin{cases} (1) \in (F^*)^1 & \mbox{if}\ \ g'_l=0\\  (1, 1/g'_l)\in (F^*)^{2}& \mbox{if}\ \ g'_l\neq 0\end{cases},$$ 
	we have $\mathcal L_\mathbf f(P)\in F[\eta_1,\dots, \eta_i]$ is of degree $\leq l-1.$ From the induction assumption on the degree it follows that $\mathcal L_\mathbf f(P)$ is an integrated integral of $F$ and this in turn  implies that $P$ is a generalized iterated integral of $F$ by Remark \ref{integralsofiteratedintegrals}. This proves that the Picard-Vessiot ring $T(E|F)$ is the set of all generalized iterated integrals in $E$ of $F.$
	
	\eqref{pvring-iae} $\Rightarrow$ \eqref{diffeq-iae}: Suppose that $T(E|F)$ is generated as an $F-$algebra by $\eta_1,\dots, \eta_n,$ where each $\eta_i$ is a generalized iterated integral of $F.$ 
    Since each $\eta_i$ is a solution of some monic linear homogeneous differential equation over $F,$ the $F-$vector space $V$ spanned by 
    $$\{\eta^{(j)}_i\ | \ i=1,2\dots, n\ \text{and} \ j=0, 1,2,\dots\}$$ 
    is a finite dimensional differential $F-$module.
    Now, the assumption that $F\neq C$ guarantees the existence of a cyclic vector (cf.~\cite{Kovacic-2001}) for $V,$ that is, $V$ is spanned as an $F-$vector space by $\eta, \eta', \dots, \eta^{(m)}$ for some $\eta \in V$ and some nonnegative integer $m.$ 
    Then, since $\eta_1,\dots, \eta_n\in V\subset T(E|F),$ we obtain that $E=F\langle V\rangle= F\langle \eta\rangle$ and that $\eta \in V\subset T(E|F)$ is also a generalized iterated integral of $F.$ This proves \eqref{diffeq-iae}.\end{proof}

\section{scalar and matrix differential equations having unipotent Galois groups}

	It is known (see \cite[Section 22]{Kol-1948}) that a monic differential operator $\mathcal L\in F[\partial]$ of order $n$ can be factored as $$\mathcal L= (\partial-a_1) (\partial-a_2)\cdots(\partial-a_n),$$ 
    where each $a_i$ belongs to a finite algebraic extension of $F$, if and only if the connected component of the differential Galois group of $\mathcal{L}$  is a solvable linear algebraic group. 
    We shall now show in the next proposition that if the differential Galois group is unipotent, then there is another way of factoring $\mathcal{L}$. 
    
\begin{theorem}\label{prop:constructiomoffactorization}
    Let $\mathcal L \in F[\partial]$ be a monic operator of order $n,$  $E$ be a Picard-Vessiot extension of $F$ for $\mathcal L$ and $V\subset E$ be the set of all solutions of $\mathcal L(y)=0$.
    \begin{enumerate}[(i)]
    \item \label{item1:constructiomoffactorization} 
    The differential Galois group $\mathscr G(E|F)$ is unipotent, if and only if
there is a tuple $\mathbf f:=(f_1,f_2,\dots,f_n)\in (F^*)^n$ and an element $f_{n+1}\in F^*$ such that $\mathcal L=\mathcal L_\mathbf f f_{n+1},$ in which case, there is a $C-$basis $v_1,v_2,\dots,v_n$ of $V$  such that   
$$v_1=\frac{1}{f_{n+1}}, v_2=\frac{1}{f_{n+1}}\int\frac{1}{f_n}, v_3=\frac{1}{f_{n+1}}\left(\int\frac{1}{f_n}\left(\int\frac{1}{f_{n-1}}\right)\right),\dots,$$ $$v_n=\frac{1}{f_{n+1}}\left(\int\frac{1}{f_n}\left(\int\frac{1}{f_{n-1}}\cdots\frac{1}{f_3}\left(\int\frac{1}{f_2}\right)\cdots\right)\right).$$

   \item \label{item2:constructiomoffactorization}  
    The differential Galois group $\mathscr G(E|F)$ is unipotent, if and only if there are elements $f_2,\dots,f_n \in F^*$
    such that $E$ is a Picard-Vessiot extension of $F$ for the matrix differential equation $Y'=AY$ where 
    \begin{equation} \label{root-matrix}
    A:=\sum^{n-1}_{i=1}\frac{1}{f_{n-i+1}}\rm{E_{i,i+1}}=
    \begin{pmatrix}0&1/f_n &0&0&\cdots&0\\
	0&0& 1/f_{n-1}&0&\cdots&0\\ \vdots&\vdots&\vdots&\ddots&\vdots&0\\  0&0& 0&\cdots&1/f_3&0\\ 0&0& 0&\cdots&0&1/f_2\\  0&0& 0&\cdots&0&0\end{pmatrix}.
    \end{equation}
        \end{enumerate}
\end{theorem}
\begin{proof}
    \eqref{item1:constructiomoffactorization}: Let $V\subset E$ be the set of all solutions of $\mathcal L(y)=0$ and suppose that $\mathscr G(E|F)$ is unipotent. 
     Then, one can choose a $C-$basis $\{v_1,v_2,\dots, v_n\}$ of $V$ so that the faithful representation $\mathscr G(E|F)\hookrightarrow \mathrm{GL}(V)$  is equivalent to a representation $\mathscr G(E|F)\hookrightarrow \mathrm{GL}_n(C)$ such that the image is contained in the group of all unipotent upper triangular matrices $U(n,C)$ of $\mathrm{GL}_n(C)$, i.e.~all upper triangular matrices with $1s$ on the diagonal. Thus, for any $\sigma\in \mathscr G(E|F),$ we obtain 
    \begin{align}
    \sigma(v_1)&=v_1, \label{solninF}\\
	\sigma(v_2)&=v_2+v_1 c_{1,2}, \label{antiderivative}\\
	\sigma(v_3)&=v_3+v_2 c_{2,3}+v_1 c_{1,3}, \label{iter-antiderivaive-2}\\
	& \quad  \vdots \notag \\
    \sigma(v_n)&=v_n+v_{n-1}c_{n-1,n}+\cdots+v_2c_{2,n}+v_1c_{1,n}  \label{iter-antiderivative-n}
\end{align}
with $c_{i,j} \in C$.
Equation \eqref{solninF} implies that $v_1\in F$ and therefore $f_{n+1}:=1/v_1 \in F^*.$ From Equation \eqref{antiderivative} we obtain that $\sigma(f_{n+1}v_2)=(f_{n+1}v_2) + c_{1,2}$ and therefore 
$\sigma\left((f_{n+1}v_2)'\right)=\sigma\left(f_{n+1}v_2\right)'=\left(f_{n+1}v_2\right)',$ which implies that $\left(f_{n+1}v_2\right)'\in F.$ 
Since $v_1$ and $v_2$ are linearly independent over $C$, we conclude that $(f_{n+1}v_2)'\neq 0.$ Let $f_n:=1/(f_{n+1}v_2)'$ and then we obtain that $$(f_n(f_{n+1}v_2)')'=0.$$ 
Next, from Equation \eqref{iter-antiderivaive-2}, we obtain  $\sigma(f_{n+1}v_3)=(f_{n+1}v_3)+(f_{n+1}v_2)c_{2,3}+c_{1,3}$. It follows that $(f_{n}(f_{n+1}v_3)')'\in F$ and since $v_1$, $v_2$, $v_3$ are $C$-linearly independent, we obtain that  
$(f_{n}(f_{n+1}v_3)')'\neq0$. Let $f_{n-1}:=1/(f_{n}(f_{n+1}v_3)')'$ and then we have that
$$(f_{n-1}(f_{n}(f_{n+1}v_3)')')'=0.$$ 
Likewise for $i=0,1,\dots, n-2,$ we shall recursively define
\begin{equation} \label{recurrsivegi}
f_{n-i}:=\frac{1}{(f_{n-(i-1)}(f_{n-(i-2)}(\cdots (f_{n+1}v_{i+2})'\cdots)')')'} \in F^* \end{equation} and obtain that 
\begin{equation}\label{diff-eqn-gii}
(f_{n-i}(f_{n-(i-1)}(f_{n-(i-2)}(\cdots (f_{n+1}v_{i+2})'\cdots)')')')'=0.
\end{equation}
In particular, for $i=n-2,$ we have  
\begin{equation} 
(f_2(f_3\cdots(f_n(f_{n+1}v_{n})')'\cdots)')'=0 .
\end{equation}
We choose an element $f_1\in F^*$ so that for \begin{equation}\label{choiceofg}\mathbf f:=(f_1,f_2,\dots, f_n),\end{equation}
the operator $\mathcal L_\mathbf f f_{n+1}$ is monic.  Since $f_{n+1}v_1=1$, we have $\mathcal L_\mathbf f f_{n+1}(v_1)= \mathcal L_\mathbf f(f_{n+1}v_1)=0$. From Equation (\ref{diff-eqn-gii}) we conclude that 
$$\mathcal L_{\mathbf f}f_{n+1}(v_{i+2})=0$$
for all $i=0,1,\dots, n-2.$  
Thus, $V$ is also the set of all solutions of the monic linear differential equation 
\begin{equation}\label{scalardiffeqL_g}
\mathcal L_{\mathbf f}f_{n+1} (y)=0.
\end{equation}
Since $\mathcal L_\mathbf f f_{n+1}$ is monic and its order is $n$, it follows that $\mathcal L= \mathcal L_\mathbf f f_{n+1}$.

The converse of \eqref{item1:constructiomoffactorization} follows from Theorem~\ref{Uni-PVring}.\\
We have that $v_1=\frac{1}{f_{n+1}}$ and from Equation (\ref{diff-eqn-gii}) we can write 
\begin{gather*}
v_2=\frac{1}{f_{n+1}}\int\frac{1}{f_n},\quad v_3=\frac{1}{f_{n+1}}\left(\int\frac{1}{f_n}\left(\int\frac{1}{f_{n-1}}\right)\right),\dots,   \\
 v_n= \frac{1}{f_{n+1}}\left(\int\frac{1}{f_n}\left(\int\frac{1}{f_{n-1}}\cdots\frac{1}{f_3}\left(\int\frac{1}{f_2}\right)\cdots\right)\right).
\end{gather*}
\eqref{item2:constructiomoffactorization}: Assume that $\mathscr G(E|F)$ is unipotent and keep the notation of \eqref{item1:constructiomoffactorization}. Then according to \eqref{item1:constructiomoffactorization}  we have $\mathcal L= \mathcal L_\mathbf f f_{n+1}$ and so the elements $w_i=v_i/v_1$ form a $C$-basis of the solution space of $\mathcal{L}_\mathbf f(y)=0$ and $E$ is generated as a differential field over $F$ by $w_1,\dots,w_n$, that is $E=F\langle w_1,\dots,w_n \rangle$. For $1\leq k\leq n$ we define 
\[
t_{1,k}:=w_k 
\]
and for $2\leq k \leq n$ and $k \leq l\leq n$ we define recursively
\[
t_{k,l} := f_{n-(k-2)} t'_{k-1,l} .
\] 
Let $t_{i,j}:=0$ for $i> j$ and observe that $T:=(t_{i,j}) \in U(n,E)$. 
A straight forward calculation shows that  
\[
T'=AT ,
\]
where $A$ is as in \eqref{root-matrix}.
 
Since $w_k=t_{1,k} \in F(t_{i,j}| 1 \leq i,j \leq n)$ and $F(t_{i,j}| 1 \leq i,j \leq n)$ is a differential field, we actually have that $E=F\langle w_1,\dots,w_n \rangle=F(t_{i,j}| 1 \leq i,j \leq n)$.

To prove the converse assume that $E$ is a Picard-Vessiot extension of $F$ for $Y'=AY$ and let $T \in \mathrm{GL}_n(E)$ be a fundamental matrix for $Y'=AY$. Since $A$ lies in the Lie algebra of $U(n,F)$, it follows from \cite[Prop.~1.31 (1)]{MvdP03} that $\mathscr G(E|F)$ is contained in a conjugate of $U(n,C)$ in $\mathrm{GL}_n(C)$ and is therefore also unipotent. More precisely, it follows from the proof of this proposition that we can construct a Picard-Vessiot extension of $F$ for $Y'=AY$ such that there is an upper triangular fundamental matrix and the respective differential Galois group is contained in $U(n,C)$. We obtain now from \cite[Prop.~1.20 (2)]{MvdP03} that there exists $M \in \mathrm{GL}_n(C)$ such that $TM^{-1} \in U(n,E)$ and $M \mathscr G(E|F) M^{-1} \subset U(n,C)$. Thus, we can assume that $T=(t_{i,j}) \in U(n,E)$ and $\mathscr G(E|F) \subset U(n,C)$.  For $2\leq k\leq l\leq n,$ we observe that 
$t_{k,l}= f_{n-(k-2)}t'_{(k-1),l} $ and therefore, we obtain using an induction that the entries of the fundamental matrix belong to the differential field $F\langle t_{1,2}, t_{1,3},\dots, t_{1,n}\rangle.$ In particular, $E=F\langle t_{1,2}, t_{1,3},\dots, t_{1,n}\rangle.$ Since 
$$\frac{1}{ f_{n-(k-2)}}=t'_{(k-1),k},$$ 
we also obtain for $1\leq j\leq n$ that
$$(f_2(f_3\cdots(f_{n-1}(f_nt'_{1,j})'\cdots)')'=0.$$ 
We prove that $t_{1,1},\dots,t_{1,n}$ are $C$-linearly independent. Assume that there are constants $c_i$ such that $\sum_{i=1}^n c_i t_{1,i}=0.$ Taking the first derivative of this sum, we obtain
$\sum_{i=2}^n c_i t_{2,i}=0$, since $t_{1,1}=1$ and $t_{1,i}'= t_{2,i}/f_n$ for every $2\leq i \leq n$. Now taking again the derivative and using that $t_{2,2}=1$ and $t_{2,i}'= t_{3,i}/f_{n-1}$, we get $\sum_{i=3}^n c_i t_{3,i}=0$. Continuing this way, we obtain for $1\leq j \leq n$  the equation $\sum_{i=j}^n c_i t_{j,i}=0$. For $j=n$ we then have $c_n t_{n,n}=0$ and so $c_n=0$. For $j=n-1$ we then obtain from $c_{n-1} t_{n-1,n-1}+c_{n} t_{n-1,n}=0$ that $c_{n-1}=0$. Continuing in this way, we obtain that $c_n=c_{n-1}=\dots=c_1=0$ showing that $t_{1,1},\dots,t_{1,n}$ are $C$-linearly independent.
Choosing a suitable $f_1,$ we conclude that $E$ is also a Picard-Vessiot extension for the monic linear differential operator $\mathcal L_\mathbf f\in F[\partial],$ where $\mathbf f=(f_1,f_2,\dots, f_n)\in (F^*)^{n}.$
\end{proof}

\begin{corollary}\label{corollary}
    Let $\mathbf{f}=(f_1,\dots,f_n) \in (F^*)^n$ be an element such that $\mathcal{L}_{\mathbf{f}}$ is a monic operator.
    Then a differential field extension $E$ of $F$ is a Picard-Vessiot extension of $\mathcal{L}_{\mathbf{f}}$, if and only if it is a Picard-Vessiot extension of $F$ for $Y'=AY$ with 
  \[A=\sum^{n-1}_{i=1}\frac{1}{f_{n-i+1}}\rm{E_{i,i+1}}.\]  
\end{corollary}
\begin{proof}
Readily follows from the proof of Theorem \ref{prop:constructiomoffactorization}
\end{proof}

		The rest of this section is dedicated to construct for $F=C(x)$ and for a given unipotent group $U$ an element ${\bf f} \in (F^*)^n$ such that the differential Galois group for $\mathcal{L}_{\bf f}(y)=0$ or for $Y'=AY$ with $A$ is in \eqref{root-matrix} is $U(C)$.

        Let $\mathfrak u(n,C)$ be the lie algebra of $U(n,C)$, that is the Lie algebra of all upper triangular matrices with zeros on the diagonal. 
		Let $U\subseteq U(n,C)$ be a unipotent matrix group of dimension $m$ and let $\mathfrak u\subseteq \mathfrak u(n,C)$ be its Lie algebra. Moreover, let $C[Z_{i,j}\ | \ 1\leq i<j\leq n]$ be the coordinate ring of $U(n,C)$ and let 
        $$I\lhd C[Z_{i,j}\ | \ 1\leq i<j\leq n]$$ 
        be the defining ideal of $U.$  The commutator subgroup $[U, U]$ of $U$ is normal and we can consider the quotient homomorphism $\pi:U \rightarrow U/[U,U]$.  Let $d\pi: \mathfrak u\to \mathrm{Lie}(U/[U,U])$ be the induced homomorphism of Lie algebras. If $l$ is the dimension of the commutative group $U/[U,U]$, we can choose a basis $X_1,\dots,X_{l},\dots,X_m$ of $\mathfrak{u}$ such that 
	\[d\pi(X_1),\dots,d\pi(X_{l})\] form a basis of $\mathrm{Lie}(U/[U,U]).$ 
	
	Let $\tilde{a}_1,\dots,\tilde{a}_{l}$ be rational functions in $F$ such that their images in $F/F'$ are $C$-linearly independent and define 
	\[
A_{\mathfrak{u}}:= \sum_{i=1}^{l} \tilde{a}_i X_i \in \mathfrak{u}(F)\subseteq \mathfrak u(n,F).\]
Then, it follows from \cite[Proposition 14]{KovacicInvProb} and \cite[Proposition 11]{KovacicInvProb} that the differential Galois group of any Picard-Vessiot extension $E$ of $Y'=A_u Y$ is isomorphic to $U(C).$ 
 We are going to construct a Picard-Vessiot extension $E$ of $F$ for $Y'=A_u Y$.
We extend the derivation of $F$ to the polynomial ring 
$$R:=F[Z_{i,j} \mid 1\leq i<j \leq n]$$ 
in the indeterminates $Z_{i,j}$ over $F$ by $Z'=A_{\mathfrak{u}}Z$, where $Z$ is the matrix 
\begin{equation}\label{genericpoint}Z=\begin{pmatrix}1&Z_{1,2}&Z_{1,3}&\cdots&Z_{1,n}\\ 0&1&Z_{2,3}&\cdots&Z_{2,n}\\ \vdots&\vdots&\vdots&\vdots&\vdots\\ 0&0&\cdots&1&Z_{n-1,n}\\ 0&0&\cdots&0&1\end{pmatrix}\in U(n, R).
\end{equation}  
We consider now the ideal $I^e$ in $R$, which is generated by the defining ideal $I$ of $U$. Since $A_{\mathfrak{u}}$ is an element of $\mathfrak{u}(F)$, it follows from the proof of \cite[Proposition 1.31 (1)]{MvdP03} that the ideal $I^e$ is a differential ideal. Choosing a maximal differential ideal $\mathfrak m$ in $R$ containing $I^e$ we obtain a Picard-Vessiot extension
 $$E:=\mathrm{Frac}\left(R / \mathfrak m\right)$$ of $F$ for the equation $Y'=A_\mathfrak u Y.$ 
 
 Since the differential Galois group of any Picard-Vessiot extension for $Y'=A_\mathfrak u Y$ is isomorphic to $U(C)$, we conclude that $\mathscr G(E|F) \cong U(C)$.
It follows from \cite[Corollary 5.29]{mag-book} that there is a $\mathscr G(E|F)$ equivariant $F-$algebra isomorphism 
$$T(E|F)\cong  F\otimes_C C[U] \cong R/I^e,$$
where $C[U]:=C[Z_{i,j} \mid 1\leq i<j \leq n]/I$ denotes the coordinate ring of $U$. It is also known that $T(E|F)=R/\mathfrak m$ (see \cite[Corollary 1.38]{MvdP03}) as differential $F-$algebras and so $R/\mathfrak{m} \cong R/I^e$. Since $I^e $ is contained in $\mathfrak{m}$, we conclude that $I^e= \mathfrak{m}$ and so $E=\mathrm{Frac}\left(R/I^e\right)$ is a Picard-Vessiot extension of $F$ for the equation $Y'=A_\mathfrak u Y$.  

Let $\phi: R \to R/I^e$ be the quotient morphism. We denote the image of $Z_{i,j}$ under $\phi$ by $\overline{Z}_{i,j}$ and for a matrix $M \in \mathrm{GL}_n(R)$ we mean by $\phi(M)$ the matrix obtained by applying $\phi$ to the entries of $M$.
Then the matrix 
$$ 
\overline{Z}:= \phi(Z) \in U(E)
$$ 
is a fundamental matrix for the equation $Y'=A_\mathfrak u Y.$ Furthermore, for any  $\sigma\in \mathscr G(E|F)$  we have $\sigma(\overline{Z}) \in U(E)$ and $\sigma(\overline{Z})=\overline{Z} C_\sigma,$ where $C_\sigma\in  \mathrm{GL}_n(C).$ Thus, $C_{\sigma} \in U(C)$ and the representation of the differential Galois group for $\overline{Z}$ is $U(C)$. 

In order to construct a scalar differential equation of the form $\mathcal L_\mathbf f (y)=0$ having $E$ as its Picard-Vessiot extension, we first find a cyclic vector $v\in E^n$ for the differential $E-$module $E^n$ defined by the equation $Y'=A_\mathfrak u Y$ as in \cite{churchill2002cyclic}. This process provides a matrix $B=(b_{i,j})\in \mathrm{GL}(n,F)$ such that 
$$BA_\mathfrak u B^{-1}+B'B^{-1}=:A_c\in M(n,F),$$
where $A_c$ is a companion matrix
$$A_c=\begin{pmatrix}0&1&0&0&\cdots&0\\0&0&1&0&\cdots&0\\ \vdots&\vdots&\vdots&\ddots&\vdots&\vdots\\ 0&0&0&\cdots&1&0\\0&0&0&\cdots&0&1\\ -a_0&-a_1&-a_3&\cdots&\-a_{n-2}&-a_{n-1}\end{pmatrix}$$ 
with $a_0,a_1,\dots, a_{n-1}\in F.$
Next we observe that $(BZ)'=A_c BZ$ and so
$$BZ=\mathrm{Wr}(Y_1,\dots, Y_n):=\begin{pmatrix}Y_1&Y_2&\cdots&Y_n\\Y'_1&Y'_2&\cdots&Y'_n\\ \vdots&\vdots&\cdots&\vdots\\Y^{(n-1)}_1&Y^{(n-1)}_2&\cdots&Y^{(n-1)}_n\end{pmatrix}$$
is a Wronskian matrix in $Y_1:=b_{1,1}$ and $Y_j:=\sum^j_{i=1}b_{1i}Z_{i,j}$ for $2\leq j\leq n$.  
Now let
\begin{equation} \label{(1,1)entryis1} B_0= \begin{pmatrix}  Y_1^{-1} & 0 & 0  & \dots  & 0 \\   (Y_1^{-1})' &  Y_1^{-1} & 0 & & \vdots \\ (Y_1^{-1})'' &  2(Y_1^{-1})' & Y_1^{-1} \\ \vdots & \vdots & &\ddots  &  \\ & & & & \\ (Y_1^{-1})^{(n-1)} & \dots & & (n-1) (Y_1^{-1})' & Y_1^{-1} \end{pmatrix} \end{equation} and observe that $B_0BZ$ remains a Wronskian matrix. In fact, we now have 
 $$B_0BZ= \mathrm{Wr}(1,Y_2 Y_1^{-1}\dots,Y_n Y_1^{-1}).$$    
 
Let $\overline{Y}_{i,j}$ be the image of the $(i,j)$-th entry of the Wronskian matrix $B_0BZ$ under $\phi$, that is 
$$\overline{Y}_{i,j}:=\left( \phi(B_0BZ)\right)_{i,j}=\left(B_0B (\overline{Z}_{i,j})\right)_{i,j}.$$ 
Since $B_0  B\in \mathrm{GL}(n, F)$ and $E$ is generated as a field by $\{\overline{Z}_{i,j}\ | \ 1\leq i,j\leq n\}$ over $F$, it follows that $E$ is also generated as a field  over $F$ by $\{ \overline{Y}_{i,j}\ | \ 1\leq i,j\leq n\}$.  Furthermore, $E$ is a Picard-Vessiot extension of a monic operator $\mathcal L\in F[\partial]$, whose full set of solutions is the $C$-vector space $\mathrm{span}\{1,\overline{Y}_{1,2},\dots, \overline{Y}_{1,n}\}.$ 
 
 For any $\sigma\in \mathscr G(E|F),$ there is a matrix $C_\sigma=(c_{i,j}^{\sigma})\in U(C)$ such that  $\sigma(\overline{Z})=\overline{Z}C_\sigma$ and therefore $\sigma(B_0B\overline{Z})=B_0B\overline{Z}C_\sigma.$ Then, 
for each $2\leq i\leq n$ we have  
$$\sigma(\overline{Y}_{1,i})=\overline{Y}_{1,i}+\overline{Y}_{1,i-1}c_{i-1,i}^{\sigma}+\cdots+\overline{Y}_{1,2}c_{2,i}^{\sigma}+c_{1,i}^{\sigma}.$$ 
Now we construct recursively $\mathbf{f}\in (F^*)^n$ as in Equation (\ref{recurrsivegi}) with $v_i$ replaced by $\overline{Y}_{1,i}$ and consider $\mathcal L_\mathbf f\in F[\partial]$. Then the full set of solutions of
$\mathcal L_\mathbf{f}(y)=0$  is 
$$\mathrm{span}_C\{1, \overline{Y}_{1,2},\dots, \overline{Y}_{1,n}\}.$$ 
Thus, $E=F\langle \overline{Y}_{1,2},\dots, \overline{Y}_{1,n}\rangle$ is a Picard-Vessiot extension of $F$ for $\mathcal L_{\mathbf f}(y)=0$ with differential Galois group $U(C)$. 

The element $\mathbf{f}\in (F^*)^n$ can be actually computed using the generators of the defining ideal $I$ of $U$. More precisely, applying the recursion defined in Equation (\ref{recurrsivegi}) with $v_i$ replaced by $Y_iY^{-1}_1$  
we obtain 
\begin{equation}\label{eqn:definingG_i}
\begin{array}{l}\displaystyle
G_n:=\frac{1}{(Y_2Y^{-1}_1)'},\  G_{n-1}:=\frac{1}{(G_n(Y_3Y^{-1}_1)')'}, \  \dots, \\[0.5em]
\displaystyle G_2:=\frac{1}{(G_3(G_4(\cdots G_{n-1}(G_n(Y_nY^{-1}_1)')')'\cdots)'}.
\end{array}
\end{equation}
Then, $f_i=\phi(G_i)$ for each $1\leq i\leq n$, since $Y_iY^{-1}_1$ is a preimage of $\overline{Y}_{1,i}$ under $\phi$ and we applied the same recursion.  
Let $N_i$ and $D_i \in R$ such that $G_i=N_i/D_i$. Then one computes the normal form of $N_i$ and $D_i$ with respect to a reduced Gröbner basis for $I^e$ and shortens its fraction to obtain $\phi(G_i)=\phi(N_i)/\phi(D_i) \in F.$

We shall illustrate the above procedure through the following examples.
\begin{example} 
The set $\{\mathrm{E_{i,j}}\ | \ 1\leq i<j\leq n\}$ forms a $C$-basis of the lie algebra $\mathfrak{u}(n, C)$ and it is easily seen that the set $\{d\pi(\mathrm{E}_{1,2}), d\pi(\mathrm{E}_{2,3}),\dots, d\pi(\mathrm{E}_{n-1,n})\}$ forms  a $C-$basis of the Lie algebra of $$U(n, C)/[U(n,C), U(n,C)].$$
For pairwise distinct constants $c_2,\dots,c_n \in C$ define   
	\[
	 f_{2}:=x-c_2 ,\dots, f_n:=x-c_n \in F .
	\]
    Since the derivative of any element of $F$ has no simple poles in $C,$ it can be seen that  
    $\sum^{n}_{i=2}\lambda_i\frac{1}{f_i}$ is a derivative of an element of  $F$ if and only if $\lambda_i=0$ for all $i.$ From  \cite[Proposition 14]{KovacicInvProb} and \cite[Proposition 11]{KovacicInvProb}, it follows that the differential Galois group of the differential equation 
    $$Y'=AY \qquad \text{with} \quad  A:=\sum^{n-1}_{i=1}\frac{1}{f_{n+1-i}}E_{i,i+1}$$ 
    is isomorphic to $U(n, C).$ Equivalently, one chooses $f_1\in F^*$ so that for $\mathbf f=(f_1,f_2,\dots,f_n)\in (F^*)^n$ the operator $\mathcal L_\mathbf f\in F[\partial]$ is monic. We obtain that the differential Galois group of $\mathcal L_\mathbf{f}(y) =0$ is $U(n,C)$.
\end{example}	
		
		\begin{example}
			Let $U$ be the closed subgroup of $U(3,C)$ defined by the ideal $I:=\langle Z_{2,3}\rangle$ in the coordinate ring $ C[Z_{1,2}, Z_{1,3}, Z_{2,3}]$ of $U(3,n)$. Then 
             $$U=\left\{\begin{pmatrix}1&c_{1,2}&c_{1,3}\\0&1&0\\0&0&1\end{pmatrix}\ | \ c_{1,2}, c_{2,3}\in C\right\} \quad \text{and} \quad \mathfrak u=\left\{\begin{pmatrix}0&c_{1,2}&c_{1,3}\\0&0&0\\0&0&0\end{pmatrix}\ | \ c_{1,2}, c_{2,3}\in C\right\}.
            $$ 
            Since  $[U, U]=1,$  the induced map $d\pi: \mathfrak u\to \mathrm{Lie}(U/[U,U])$ is an isomorphism. Thus,  we choose the basis $\mathrm{E}_{1,2}$ and $\mathrm{E}_{1,3}$ of $\mathfrak u$ and elements $\frac{1}{x-c_1}, \frac{1}{x-c_2}\in F,$ where $c_1$ and $c_2$ are distinct constants, and let    
            $$ A_\mathfrak u :=\begin{pmatrix}0&\frac{1}{x-c_1}&\frac{1}{x-c_2}\\0&0&0\\0&0&0\end{pmatrix}.$$
Then, the derivatives of the indeterminates $Z_{1,2}, Z_{1,3}, Z_{2,3}$ in $R=F[Z_{1,2}, Z_{1,3},Z_{2,3}]$ are
\begin{equation}\label{eqn:derivativesofZ_ij}
Z'_{1,2}=\frac{1}{x-c_1}, \ Z'_{1,3}=\frac{1}{x-c_2} \ \text{and} \ Z'_{2,3}=0. 
\end{equation}
Moreover, the above discussion implies that $E:=\mathrm{Frac}(R/I^e)$ is a Picard-Vessiot extension of $F$ for $Y'=A_\mathfrak u Y$ with the differential Galois group $U(C)$. 

Consider the differential module $(E^3,\delta)$ with standard basis $\{ \mathrm{e}_1, \mathrm{e}_2,\mathrm{e}_3\}$ and derivatives   
$$\delta(\mathrm{e}_i)=\sum^3_{j=1}(A_\mathfrak{u})_{(i,j)}\mathrm{e}_j \quad \text{for} \ i=1,2,3.$$ 
We compute
$$\delta(\mathrm{e}_2)=\delta(\mathrm{e}_3)=0,\quad \delta(\mathrm{e}_1)=\frac{1}{x-c_1}\mathrm{e}_2+\frac{1}{x-c_2}\mathrm{e}_3,\quad \delta^{2}(\mathrm{e}_1)=\frac{-1}{(x-c_1)^2}\mathrm{e}_2+\frac{-1}{(x-c_2)^2}\mathrm{e}_3$$
and conclude that $\mathrm{e}_1$ is a cyclic vector. Then 
$$B:=\begin{pmatrix}1&0&0\\ 0& \frac{1}{x-c_1}& \frac{1}{x-c_2}\\ 0& \frac{-1}{(x-c_1)^2}& \frac{-1}{(x-c_2)^2}\end{pmatrix}
\quad \text{satisfies} \quad  \begin{pmatrix}\rm e_1\\ \delta(\rm e_1)\\ \delta^2(\rm e_1)\end{pmatrix}=B \begin{pmatrix}\rm e_1\\ \rm e_2\\ \rm e_3\end{pmatrix}.$$ 
Since the first entry of $BZ$ is $1,$ the matrix $B_0$ is here the identity matrix and so 
$$B_0BZ=\begin{pmatrix}1& Z_{1,2}&Z_{1,3}\\ 0& \frac{1}{x-c_1}& \frac{1}{x-c_1}Z_{2,3}+\frac{1}{x-c_2}\\ 0& \frac{-1}{(x-c_1)^2}& \frac{-1}{(x-c_1)^2}Z_{2,3}+\frac{-1}{(x-c_2)^2}\end{pmatrix}=:\mathrm{Wr}(1,Y_2Y_1^{-1},Y_2Y_1^{-1}).$$

Using \eqref{eqn:derivativesofZ_ij} the recursion in \eqref{eqn:definingG_i} with $Y_2Y_1^{-1}=Z_{1,2}$ and $Y_3Y_1^{-1}=Z_{1,3}$ becomes
$$G_3=\frac{1}{Z'_{1,2}}=x-c_1 \quad \text{and} \quad G_2=\frac{1}{\left(
	\frac{Z'_{1,3}}{Z'_{1,2}}\right)'}=\frac{1}{\left(Z_{2,3}+\frac{x-c_1}{x-c_2}\right)'}.$$
Reducing the numerators and denominators of $G_1$ and $G_2$ with respect to $Z_{2,3}$ we obtain  
$f_3:=\phi(G_3)=x-c_1 $ and $f_2:=\phi(G_2)=\frac{(x-c_2)^2}{c_1-c_2}$, where as above $\phi$ denotes the projection 
$\phi: R \to R/I^e$. Thus, $\mathrm{span}_C\{1, \overline{Z}_{1,2},\overline{Z}_{1,3}\}$ is a full set of solutions for the differential operator 
$$\partial f_2\partial f_3\partial=\frac{(x-c_2)^2(x-c_1)}{c_1-c_2}\partial^3+\frac{\left(2(x-c_2)^2+2(x-c_2)(x-c_1)\right)}{c_1-c_2}\partial^2+\frac{2(x-c_2)}{c_1-c_2}\partial \in F[\partial].$$
Choosing $f_1=\frac{1}{(x-c_2)^2(x-c_1)}$ and defining $\mathbf{f}=(f_1,f_2,f_3)$ we obtain that $E=F\langle \overline{Z}_{1,2},\overline{Z}_{1,3}\rangle$ is a Picard-Vessiot extension for the monic linear differential equation 
$$\mathcal{L}_\mathbf{f}(y)=y'''+2\left(\frac{1}{x-c_1}+\frac{1}{x-c_2}\right)y''+\frac{2}{(x-c_2)(x-c_1)}y'$$ as well as for the matrix differential equation
 $$Y'=\begin{pmatrix}0&\frac{1}{x-c_1}&0\\0&0&\frac{c_1-c_2}{(x-c_2)^2}\\0&0&0\end{pmatrix}Y.$$
\end{example} 

We summarize our results in Algorithm~\ref{alg:constructshapematrix} which computes a matrix differential equation $Y'=AY$ with $A$ as in \eqref{root-matrix} having a given group $U(C)$ as its differential Galois group. Using Corollary \ref{corollary} one can calculate from the entries of $A$ the scalar differential equation of the form $\mathcal L_\mathbf f(y)=0,$ whose differential Galois group is also $U(C)$. 
\begin{algorithm}
     \DontPrintSemicolon
\KwInput{
\begin{enumerate}
    \item The generators of the defining ideal $I \lhd C[Z_{i,j},\mathrm{det}(Z_{i,j})^{-1}]$ of $U(C) \subset U(n,C)$,
    \item the first basis vectors $X_1,\dots,X_l$ of a basis of $\mathfrak{u}$ such that $d\pi(X_1),\dots, d\pi(X_l)$ forms a basis of the Lie algebra of $(U/[U,U])(C)$,
    \item rational functions $\tilde{a}_1,\dots,\tilde{a}_l$ in $F$ such that their images in $F/ F'$ are $C-$linearly independent.
\end{enumerate}
}
\KwOutput{A matrix $A$ as in Corollary \ref{corollary} such that the differential Galois group of $Y'=AY$ is $U(C)$.
}
Define $A_{\mathfrak{u}}:=\sum_{i=1}^{l} \tilde{a}_i X_i$.\\
Compute a cyclic vector for $Y'=A_{\mathfrak{u}}Y$, that is a matrix $B \in \mathrm{GL}_n(F)$ such that 
\[B A_{\mathfrak{u}}B^{-1}+B'B^{-1}=:A_c\] is a companion matrix.\\

Let $Z$ be as in \eqref{genericpoint} and $B_0$ as in \eqref{(1,1)entryis1} with $Y_1$ replaced by $(BZ)_{1,1}$. Compute the matrix 
$$B_0 BZ.$$

Compute  in $\mathrm{Frac}(R)$ recursively the elements $G_{n},\dots,G_{2}$ as in \eqref{eqn:definingG_i} with $Y_2Y_1^{-1},\dots,Y_n Y_1^{-1}$ replaced by $(B_0 BZ)_{1,2},\dots,(B_0 BZ)_{1,n}$.\\

Let $N_i$, $D_i \in R$ such that $G_i=N_i/D_i$. Compute the normal forms of $N_i$ and $D_i$ in $R$ with respect to a reduced Gr\"obner basis of $I^e$ and denote the shortened fraction of the normal forms of $N_i$ and $D_i$ by $f_i$.

Define the matrix $$A=\sum_{i=1}^{n-1} \frac{1}{f_i} E_{i,i+1} $$
\Return{$A$}
\caption{}\label{alg:constructshapematrix}
\end{algorithm}
Note that for the construction of ${\bf f} \in (F^*)^n$ Algorithm~\ref{alg:constructshapematrix} uses the generators of the defining ideal $I$ of $U$ and a basis of the Lie algebra. We want to point out that it is sufficient that $U$ is given either by generators of its defining ideal $I$ or by a basis of its Lie algebra, since one can compute one from the other. Clearly, given generators of $I$ one can compute the defining linear equations of the tangent space of $U$ at the unit matrix whose solution space is the Lie algebra. If basis elements $X_1,\dots,X_m$ of the nilpotent Lie algebra $\mathfrak u$ of $\mathfrak u(n,C)$ are given, then we can compute generators of the defining ideal $I\lhd C[Z_{ij}| 1\leq i<j\leq n]$ of $U$ using 
$$\mathrm{exp}(x_1X_1+\dots+x_mX_m )$$ 
and Gr\"obner basis methods (cf.~\cite{cox2008ideals}),  where $\mathrm{exp}: \mathfrak u\to U$ is the surjective exponential map and $x_1,\dots,x_m$ are indeterminates over $C$.

	\section{Stable Iterated Integrals}
	We shall apply now Theorem \ref{antiderivative-PV-commutative} to study stability problems in integration in finite terms.  
    \begin{definition}
    An element $g \in F$ is said to be elementary integrable over $F$, if there is an elementary extension $E$ of $F$  and an element $y \in E$ such that $y'=g.$  
	\end{definition}
	In \cite{rosenlicht1968liouville}, Rosenlicht proved Liouville's Theorem on integration in finite terms. It states that an element $g\in F$ is elementary integrable over $F$ if and only if there exist elements $u_1,\dots, u_n \in F^{*}$ and $v \in F$ and constants $c_1,\dots,c_n \in C$ such that 
    \begin{equation}\label{eqn:LiouvillesTheorem}
	g=v'+\sum^n_{i=1}c_i(u'_i/u_i).
    \end{equation}
	
	\begin{definition}
    \begin{enumerate}
        \item For an integer $n\in \mathbb{N},$ an element $g\in F$ is said to be elementary $n-$integrable over $F$, if there is an element $\eta$ in some elementary extension of $F$ such that $\eta^{(n)}=g.$  If such an element $\eta$ exists in $F$ itself, then we say $g$ is $n-$integrable in $F.$
	    \item An element $g\in F$ is said to be elementary $\infty-$integrable over $F$, if there is a sequence $\eta_1,\eta_2,\dots$ of elements with each $\eta_i$ belonging to some elementary extension of $F$ such that $\eta^{(i)}_i=g.$ If such a sequence $\eta_1,\eta_2,\dots$ of elements can be found in $F$ itself, then we say $g$ is  $\infty-$integrable in $F.$
	\end{enumerate}
    \end{definition}
	Note that if $g\in F$ is $n-$integrable in $F$ (respectively, elementary $n-$integrable over $F$)  for $n\in \mathbb N\cup\{\infty\}$, then it is also  $m-$integrable in $F$ (respectively, elementary $m-$integrable over $F$) for any $m\leq n.$ 
    The problem of finding $m-$integrable elements and $\infty-$integrable elements in a field has a strong connection to the theory of dynamical systems, as described in \cite{chen2022stability} and \cite{chenfengguo}. Our results are motivated by these articles, wherein, elements of $F$ that are $\infty-$integrable in $F$ are called the \emph{stable elements} of $F.$

	\begin{proposition}\label{existenceofxinF(y)}
		Suppose that $z'\neq 1$ for any $z \in F$. Let $F(y)$ be a differential field extension of $F$  with $y \notin F$ and $y' \in F$. Assume that $C$ is the field of constants of $F(y)$. For an element $g\in F$ and an integer $n\geq 2$ suppose that $g$ is $n-$integrable in $F(y)$, that is there exists $\eta\in F(y)$ with $\eta^{(n)}=g$. Assume further that $\eta^{(m)}\notin F$ for any $m<n.$ Then there exists an element $x\in F(y)$ with  $x'=1.$ 
        
		\end{proposition}
	\begin{proof}
		 Let $\eta\in F(y)$ with $\eta^{(n)}=g \in F$ and $n\geq 2.$ Clearly, $F(y)$ is a Picard-Vessiot extenion of $F$ with Picard-Vessiot ring $F[y]=T(F(y)|F)$. Since $\eta$ satisfies a differential equation over $F$, we obtain that $\eta \in F[y]$, say
         $$\eta=a_ky^k+a_{k-1}y^{k-1}+\cdots+a_0,$$ where $a_i \in F$ and $a_k\neq 0.$
         
         For any differential automorphism $\sigma\in \mathscr G(F(y)|F),$ we have 
         $$g=\sigma(g)=\sigma(\eta^{(n)})=\sigma(\eta)^{(n)}$$ 
         and so $(\sigma(\eta)-\eta)^{(n)}=0$.  
		If $\sigma(\eta)-\eta \in C$ for all $\sigma\in \mathscr G(F(y)|F)$, then $\sigma(\eta')=\eta'$ for all $\sigma\in \mathscr G(F(y)|F)$ implying that $\eta'\in F.$ 
        This contradicts the assumption that $\eta^{(m)} \notin F$ for any $m<n$. Thus there exists $\sigma \in  \mathscr G(F(y)|F)$ such that $\sigma(\eta)-\eta$ is not a constant.
        
		We fix this $\sigma$ and choose now the largest integer $m$ such that $(\sigma(\eta)-\eta)^{(m)}\neq 0.$ Then $1\leq m< n$ and $(\sigma(\eta)-\eta)^{(m)}$ is a nonzero constant.
		Let $$x:=\frac{(\sigma(\eta)-\eta)^{(m-1)}}{(\sigma(\eta)-\eta)^{(m)}}$$ and observe that $x\in F(y)$ with $x'=1.$		\end{proof}
	
	We prove now a generalization of Liouville's Theorem for elementary $n-$integrable elements of $F.$ 
	
		\begin{theorem}\label{genofliouville}
		Let $g\in F$ and $n$ be a positive integer. Suppose that $g$ is elementary $n-$integrable over $F$ with $\eta^{(n)}=g$ for some $\eta$ belonging to some elementary extension of $F$ and that $\eta^{(m)}\notin F$ for any $m < n.$
		\begin{enumerate}[(i)] 
        \item \label{genofliouville-xinF}  
        If there exists an element $x\in F$ such that $x'=1$, then there exist elements $f_0, u_1,u_2,...,u_m\in F$ and polynomials $f_1,\dots, f_m\in C[x]$ with $\deg_x(f_i) < n$ for all $1\leq i\leq m$ such that  
        \begin{equation}\label{liouville-typeexpression-xinF} 
        \eta=f_0+f_1\log(u_1)+\cdots+f_m\log(u_m),
        \end{equation} 
		  or equivalently,  $$g=f^{(n)}_0+\sum_{i=1}^{m}\sum_{j=1}^{n} f_i^{(n-j)}\left(\frac{u_i'}{u_i}\right)^{(j-1)}.$$
        \item \label{genofliouville-xnotinF} 
        If there is no element $z\in F$ such that $z'=1$, then there exist  elements $f,u_1,u_2,...,u_m\in F,$ an element $x$ in some differential field extension of $F\langle \eta \rangle$ with constants $C$ such that $x'=1$ and a polynomial $P\in C[x]$ of degree at most $n$ and constants $c_1,\dots, c_m$ such that 
        \begin{equation}\label{liouville-typeexpression-xnotinF}
            \eta=P+f+c_1\log(u_1)+\cdots+c_m\log(u_m),
        \end{equation} 
        or equivalently, 
        $$g=c+f^{(n)}+\sum^m_{i=1}c_i\left(\frac{u'_i}{u_i}\right)^{(n-1)}.$$
        \end{enumerate}
	\end{theorem}
		\begin{proof} 
			\eqref{genofliouville-xinF}: 
			Let $g$ be elementary $n-$integrable over $F$ with $\eta^{(n)}=g,$ where $\eta$ belongs to an elementary extension of 
            $F.$ To avoid triviality, we may assume that $z^{(n)}\neq g$ for any $z\in F.$ Since $\eta$ is an iterated integral 
            and there exists $x \in F$ with $x'=1$, it follows from Theorem \ref{antiderivative-PV-commutative} that $F\langle 
            \eta\rangle$ is Picard-Vessiot extension of $F$ with commutative unipotent differential Galois group. Then by 
            Theorem~\ref{Group-Extension} \eqref{PV-AE} it follows that  
            $$F\langle \eta\rangle=F(\eta_1,\eta_2,...,\eta_r),$$ 
            where $\eta_1,\dots,\eta_r$ are algebraically independent over $F$ with each $\eta'_i\in F.$  Since each $\eta_i$ 
            belongs to an elementary extension of $F,$ it follows from Liouville's Theorem, more precisely from the integration of 
            \eqref{eqn:LiouvillesTheorem}, that for each $i\in \{1,\dots,r\}$ there exist an integer $m_i\geq 0,$ elements 
            $u_{i,1},\dots,u_{i,m_i}, v_i\in F$ and constants $c_{i,1},\dots, c_{i,m_i}$ such that   
            $$\eta_i=v_i+\sum_{j=1}^{m_i}c_{i,j}\log(u_{i,j}).$$ 

            Let $\{u_1,\dots, u_m\}\subseteq \{u_{i,j}\ | \ 1\leq i\leq r, \  1\leq j\leq m_{i}\}$ be one of the largest subset such that 
			$$\log(u_1), \log(u_2), \dots ,\log(u_m)$$ 
            are algebraically independent over $F$. 
            It then follows from \cite{kolchin1968algebraic} that 
            $$F\langle \eta\rangle=F(\eta_1,\eta_2, \dots ,\eta_r)\subseteq F(\log(u_1),\log(u_2),\dots ,\log(u_m)).$$ 
            
	        Since  $F[\log(u_1),\dots, \log(u_m)]$ is the Picard-Vessiot ring of $F(\log(u_1),\log(u_2),\dots ,\log(u_m)),$ 
            there exist  $a_0, \dots, a_l\in F[\log(u_1),\log(u_2),\dots,\log(u_{m-1})]$ with $a_l\neq 0$ such that
            \begin{equation}\label{eqn:representationetainlog(u_m)}
            \eta=a_0+a_1\log(u_m)+ \dots +a_l \log^l(u_m).
            \end{equation}
  
 Since $\eta^{(n)}=g\in F,$ the $C-$vector space $\{\eta,1, x,\dots, x^{n-1}\}$, which is the solution space of the respective differential equation \eqref{diffeqn-ItInt} for $g$, is stabilized by $\mathscr G(F\langle \eta\rangle |  F)$. The differential Galois group $ \mathscr G(F(\log(u_1),\log(u_2),\dots,\log(u_m)) |  F)$ is a commutative unipotent group and so the group homomorphism 
 \[
  \mathscr G(F(\log(u_1),\log(u_2),\dots,\log(u_m)) |  F) \rightarrow \mathscr G(F\langle \eta\rangle |  F)
 \]
 given by restriction to $F\langle \eta\rangle$ is surjective. Thus, for all $\sigma \in \mathscr G(F(\log(u_1),\log(u_2),\dots,\log(u_m)) |  F)$ there exist $P_{\sigma}(x)\in C[x]$ with $\text{deg}_x(P_{\sigma}(x))\leq n-1$ and $c_{m,\sigma}\in C$ such that 
\begin{eqnarray}
    \sigma (\eta) &=& \eta+P_{\sigma}(x) \, \label{eqn:1} \, ,  \\
    \sigma(\log(u_m))&=&\log(u_m)+c_{m,\sigma}. \label{eqn:2}
\end{eqnarray}

     Applying $\sigma$ to equation~\eqref{eqn:representationetainlog(u_m)} and then using \eqref{eqn:1}, \eqref{eqn:2} and again \eqref{eqn:representationetainlog(u_m)} to replace $\eta$, we obtain
     $$a_0+a_1\log(u_m)+\dots+a_l \log^l(u_m)+P_\sigma(x)=\sigma(a_0)+\sigma(a_1)(\log(u_m)+c_{m,\sigma})+\dots+\sigma(a_l)(\log(u_m)+c_{m,\sigma})^l.$$ 
	 
	 Comparing the coefficients of $\log^l(u_m),$ we conclude $\sigma (a_l)=a_l$ and so $a_l\in F$. If $l>1,$ then from comparing the coefficients of $\log^{l-1}(u_m),$ we obtain $lc_{m,\sigma}a_l+\sigma (a_{l-1})=a_{l-1}$. Thus 
     $\sigma (a_{l-1}/(l a_l ))= a_{l-1}/(l a_l )-c_{m,\sigma}$
     and so 
     $$\sigma \left(\log(u_m)+(a_{l-1}/(la_l))\right)=\log(u_m)+c_{m,\sigma}+(a_{l-1}/(la_l))-c_{m,\sigma}=\log(u_m)+(a_{l-1}/(la_l))$$
      for all $\sigma \in \mathscr G(F(\log(u_1),\log(u_2),\dots,\log(u_m)) |  F)$. Hence, $\log(u_m)+a_{l-1}/(la_l)\in F$ implying that $\log(u_m)\in F[\log(u_1),\dots,\log(u_{m-1})].$ This contradicts our assumption that $\log(u_1),\log(u_2),\dots,\log(u_m)$ are algebraically independent over $F.$ Thus $l\leq 1$ and so $\eta=a_0+a_1\log(u_m)$ with $a_1\in F$ and $a_0\in F[\log(u_1),\log(u_2),\dots,\log(u_{m-1})].$ 
	   
       Since the above calculations hold for each $\log(u_i)$ with $1\leq i\leq m$, we conclude that there are elements $f_0, f_1\dots,f_{m}\in F$ such that 
     \begin{equation}\label{liouvilletypeexp-eta}
     \eta=f_0+f_1\log(u_1)+\cdots+f_m\log(u_m).
     \end{equation} 
	 It is left to show that $f_i \in C[x]$ for $1\leq i \leq m$. To this end let 
     $$\tau\in \mathscr G(F(\log(u_1),\log(u_2),\dots,\log(u_{m}))| F(\log(u_1),\log(u_2),\dots,\log(u_{m-1}))$$ 
     be a nontrivial differential automorphism. Then, $\tau (\log(u_i))=\log(u_i)$ for all $1\leq i\leq m-1$ and $\tau (\log(u_m))=\log(u_m)+c_{m,\tau}$ with $c_{m,\tau}\neq 0$ and so it follows from \eqref{liouvilletypeexp-eta} that 
     $\tau$ fixes $\eta-f_m \log(u_m).$ 
     As above (cf.~\eqref{eqn:1}) there exists a polynomial $P_{\tau}(x) \in C[x]$ with $\deg(P_{\tau}(x) ) \leq n- 1$ 
     such that $\tau(\eta)=\eta+P_{\tau}(x).$ 
     Hence, applying $\tau$ to $\eta-f_m \log(u_m)$ we obtain 
     $$\eta+P_{\tau}(x)-f_m(\log(u_m)+c_{m,\tau})=\eta-f_m \log(u_m).$$
	  Thus, $f_m=c^{-1}_{m,\tau}P_\tau(x)$ is a polynomial in $C[x]$ with $\deg(f_m) \leq n-1.$ Similarly, for each $1\leq i\leq m,$ we may choose an appropriate differential automorphism and conclude that $f_i$ is a polynomial in $C[x]$ of degree less than $n$. This proves \eqref{genofliouville-xinF}.

(\ref{genofliouville-xnotinF}): Assume that there is no element $z\in F$ with $z'=1.$ We prove now by induction on $n$ that there 
are constants $c_1,\dots, c_m$ and elements $f, u_1,\dots, u_m \in F$ and a polynomial $P\in C[x]$ of degree  $\leq n$ where $x$ is in some differential field extension of $F\langle\eta \rangle$ with $x'=1$ such that 
\begin{equation} \label{inductionstatement}
\eta =f+c_1\log(u_1)+\cdots+c_m\log(u_m)+P.
\end{equation}
Note that for $n=1$ this follows from Liouville's Theorem, more precisely from integrating Equation~\eqref{eqn:LiouvillesTheorem}.
Let $n>1$ and assume that the statement is true for $n-1$. Since $(\eta')^{(n-1)}=\eta^{(n)}=g\in F$ the induction assumption implies that there are constants $e_1,\dots, e_l$ and elements $h, g_1,\dots, g_l\in F$ and $P_1\in C[x]$ of degree $\leq n-1$  such that  
\begin{equation}\label{inductionassumption}
\eta'=h+e_1\log(g_1)+\cdots+e_l\log(g_l)+P_1.
\end{equation} 
Let $y=e_1\log(g_1)+\dots+e_l \log(g_l)$ and observe that $y'=e_1(g'_1/g_1)+\dots+e_l(g'_l/g_l)\in F$. We consider now $\eta'-P_1=y+h$ and we will show that 
\begin{equation}\label{eqn:1234}
y+h=P_2+f'+c_1(u'_1/u_1)+\cdots+c_m(u'_m/u_m) 
\end{equation}
for some constants $c_1,\dots,c_m,$ elements $f,u_1,\dots,u_m\in F$ and a polynomial $P_2\in C[x]$ with $\deg(P_2) \leq n-1$. Integration then gives us \eqref{inductionstatement}, where $P$ in \eqref{inductionstatement} is such that $P'=P_1+P_2$. 

Let $Q\in C[x]$ such that $Q'=P_1$. 

First, assume that $y\in F$. Since $\eta'\notin F$, it follows that $\deg(P_1)\geq 1$. This implies that $P_1\in F\langle\eta'\rangle$ and so that $x\in F \langle \eta'\rangle$. Since $(\eta-Q)'\in F$ and $\eta-Q$ belongs to an elementary extension of $F$, we can apply Liouville's Theorem to $(\eta-Q)'=\eta'-P_1$, which shows the existence of the respective elements such that equation \eqref{eqn:1234} holds.

 Assume now that $y\notin F.$  Then $y$ is transcendental over $F$ and we have $\eta'-P_1=y+h\in F[y].$

 If $\deg(P_1)\geq1$, then $P_1$ belongs to an elementary extension of $F(y)$ and so does $x$. We first apply Liouville's Theorem for $F(y)$ to $(\eta-Q)'$ and obtain an integer $m\geq 1,$  $\mathbb Q-$linearly  independent constant $c_1,\dots, c_m,$  nonzero elements $u_1,\dots, u_m\in F(y)$ and an element $\gamma \in F(y)$ such that $$\gamma'+\sum^m_{i=1}c_i\frac{u'_i}{u_i}=(\eta-Q)'.$$    
 Since $(\eta-Q)'=y+h\in F[y],$ we further deduce  from (\cite[Lemma]{rosenlicht1968liouville}) that each $u_i\in F.$
 Thus, $\gamma^{(n)}\in F,$ where $\gamma\in F(y).$ If $r\leq n$ is the smallest integer with $\gamma^{(r)}\in F$, then $r\geq 2$ as $y\notin F$. We can now apply Proposition (\ref{existenceofxinF(y)}) to $\gamma$ which shows that there exists $\tilde{x} \in F(y)$ with $\tilde{x}'=1$. Since $x$ lies in some differential field extension of $F\langle \eta \rangle$ with constants $C$, which contains $F(y)$, we conclude that $\tilde{x}=x+c$. Thus, $x$ is an element of $F(y)$. From Proposition~\ref{itofF(x)} applied to $\gamma$ we obtain  
 $$\gamma =P+f,$$ 
 where $P\in C[x]$ is a polynomial of degree at most $n$ and $f\in F.$ This shows the existence of the respective elements such that equation~\eqref{eqn:1234} holds.  

If $\deg(P_1)=0$, we repeat the same argumentation as in case $\deg(P_1)\geq 1$ with $h$ replaced by $h+P_1$ and $Q$ replaced by $0$. From this discussion, we obtain that 
\[
\eta'=\gamma' + \sum^m_{i=1}c_i\frac{u'_i}{u_i}
\]
with $\gamma=P+f$, where $f \in F$ and $P \in C[x]$ is a polynomial of degree at most $n$.
 
 \end{proof}


	
	\section{$n-$integrable and $\infty-$integrable elements of certain differential fields}
	
	\begin{example}
		Let  $F=\mathbb C(x)$ with  $x'=1.$ Then, for every $n\in \mathbb N$, we have 
		$$\frac{1}{x}=\sum_{i=1}^{n}\binom ni \left(\frac{x^{n-1}}{(n-1)!}\right)^{(n-i)}\left(\frac{1}{x}\right)^{(i-1)}=\left(\frac{x^{n-1}}{(n-1)!}\log(x)\right)^{(n)}$$ and thus $1/x$ is elementary $n-$integrable over $C(x).$ Writing elements of $F$ in its partial fraction expansion form, one can easily conclude that $f\in F$ is $\infty-$integrable in $F$ if and only if $f\in C[x].$ 
	\end{example}

	
	\begin{proposition} \label{inftinteg-simpleextn}
		Let $F(y)$ be a differential field extension of $F$ with $y$ transcendental over $F$ and assume that the constants of $F(y)$ are $C.$ Denote by $\bar{F}$ the algebraic closure of $F$.
		\begin{enumerate}[(i)]	\item \label{firstordercase} 
        If $y'= ay+b$ with $a$, $b \in F$ and $(y-c)'/(y-c)\notin F$ for any $c\in \bar{F}$, then $\infty-$integrable elements of $F(y)$ belong to $F[y].$ \\
		\item \label{expcase} If $y'=ay$ for some $0\neq a\in F$, then $\infty-$integrable elements of $F(y)$ belong to $F[y, y^{-1}].$
		\end{enumerate}
		\end{proposition}
	
	\begin{proof}
		Let $g\in F(y)$ be $\infty-$integrable and for $n \in \mathbb{N}$ let $\eta \in F(y)$ such that $g=\eta^{(n)}$.  For $c\in \bar{F}$ we can embed $\bar{F}(y-c)$ in $\bar{F}((y-c))$ as differential fields. Then, there exist $m \in \mathbb{Z}$ and elements $0\neq \lambda_m ,\lambda_{m+1},\dots \in \bar {F}$ such that 
		\begin{equation}\label{eqn:powerseries}
        \eta =\lambda_m (y-c)^m+ \lambda_{m+1}(y-c)^{m+1} + \dots \in \bar{F}((y-c)) \, .
        \end{equation}
        Assume that $y'=ay+b$ with $a, b\in F.$ Then for any $c$, $\lambda \in \bar{F}$ and $m \in 
        \mathbb{Z}$ we observe that
        \begin{equation}\label{eqn:derivativeforpowerseries}
        \left(\lambda(y-c)^m\right)'=m\lambda(ac+b-c')(y-c)^{m-1}+m\lambda a(y-c)^m+\lambda'(y-c)^m  
        \end{equation}
        and so $\mathrm{ord}_c\left(\left(\lambda(y-c)^m\right)'\right)\geq m-1.$
        Moreover, if $m=\mathrm{ord}_c(\eta)\geq 0,$ then we can prove inductively that $\mathrm{ord}_c(\eta^{(i)})\geq 0$ for any $i\in \mathbb N.$ In particular, we have that $\mathrm{ord}_c(g)\geq 0.$

With these observations, we shall move on to prove (\ref{firstordercase}) and (\ref{expcase}).
		
		(\ref{firstordercase}): We only need to show that $g$ has no poles at any $c\in \bar{F}.$ 
        Assume on the contrary that $g$ has a pole at $c\in \bar{F}$ with $\mathrm{ord}_c(g)=r <0$. Since $g$ is $\infty$-integrable, we can choose $\eta\in F(y)$ such that $g=\eta^{(|r|)}$. Because  $\mathrm{ord}_c(g)<0$, we must have $\mathrm{ord}_c(\eta)=m<0$, since otherwise the earlier observations would imply that $\mathrm{ord}_c(g)\geq0.$ Since by assumption $(y-c)'/(y-c)\notin F$ for any $c \in \bar{F}$ we get $ac+b-c'\neq 0$ and we conclude by taking the derivative of \eqref{eqn:powerseries} and using \eqref{eqn:derivativeforpowerseries} that $\mathrm{ord}_c(\eta')=m-1.$ We can continue now inductively and show that $\mathrm{ord}_c\left(\eta^{(|r|)}\right)=m+r$. We then obtain
        $$r={\mathrm{ord}}_c(g)=\mathrm{ord}_c\left(\eta^{(|r|)}\right)=m+r$$ 
        implying that $m=0.$ But this is a contradiction to $\mathrm{ord}_c(\eta)=m<0$ and so $g$ has no poles.

(\ref{expcase}):  We only need to prove that $g$ has no poles at any $0\neq c\in \bar{F}.$ 
Assume on the contrary that $g$ has  a pole at $0\neq c\in \bar{F}$ with $\mathrm{ord}_c(g)=r <0.$ 
 Since $g$ is $\infty$-integrable, we can choose $\eta\in F(y)$ such that $g=\eta^{(|r|)}$. 
 As above $\mathrm{ord}_c(g)<0$ implies that $\mathrm{ord}_c(\eta)=m<0$.
 We prove now that $ac-c' \neq 0$. Assume the contrary is true. Then $(y/c)'=(y'c-yc')/c^2$ would be zero and so $y/c$ a constant. But since $C_{\bar{F}(y)}=C_{F(y)}$, we would have that $y \in \bar{F}$ which contradicts the assumption that $y$ is transcendental over $F$. Using $ac-c' \neq 0$ we conclude similarly as above that $\mathrm{ord}_c(\eta^{(|r|)})=m+r$ and so that $m=0$. But this is a contradiction to $\mathrm{ord}_c(\eta)=m<0$. This implies that $g$ has no pole other than zero and so $g \in F[y,y^{-1}]$.

\end{proof}

	We conclude this article with Proposition~\ref{prop:inftyintegrablespecialcases} below on finding $\infty-$integrable elements of certain differential fields. In particular, Proposition~\ref{prop:inftyintegrablespecialcases} \eqref{nthrootextn} is a special case of an open problem in the theory of stable elementary integrals (see \cite[Section 5]{chen2022stability}).
    30
	\begin{proposition}\label{prop:inftyintegrablespecialcases}
		Let $F=C(x)$ be the rational function field with $x'=1$ and let $E$ be a differential field extension of $F$ having $C$ as its field of constants. Moreover, let $g \in E.$ 
		\begin{enumerate}[(i)]	
			\item\label{exptower-infint} If $E=C(x,e^x)$, then $g$ is $\infty-$integrable in $E$ if and only if $g = \sum_{i=m}^{n} f_i (e^x)^i$ for some $m,n \in \mathbb{Z}$ and $f_i\in C[x]$.
			\item \label{antitower-inftint} If $E=C(x, log(x))$, then $g$ is $\infty-$integrable in $E$ if and only if $g= \sum_{i=0}^{n} f_i (log(x))^i$ for some $n\in \mathbb{Z}_{\geq 0}$  and $f_i\in C[x,x^{-1}]$. 
			\item \label{nthrootextn} If $E=C(x^{1/n})$ for an $n\in \mathbb{N}_{\geq 2}$, then $g$ is $\infty-$integrable in $E$ if and only if $g=\sum_{i=0}^{n-1} f_i x^{i/n}$ for some $n\in \mathbb{Z}_{\geq 0}$, $f_0\in C[x]$ and $f_i\in C[x,x^{-1}]$ for all $i\geq1$.
		\end{enumerate}  
	\end{proposition}
	\begin{proof}		
	(\ref{exptower-infint}): Let $g$ be $\infty-$integrable over $F(x,e^x).$ First, we replace the field $F$ in Proposition \ref{inftinteg-simpleextn}  with $F(e^x)$ and obtain from part \eqref{firstordercase} that $g\in F(e^x)[x].$ Next, we apply Proposition \ref{inftinteg-simpleextn} with $F$ replaced by $F(x)$ and obtain from part \eqref{expcase}  that $g\in F(x)[e^x,e^{-x}].$ Using $g\in F(e^x)[x]$ we conclude that $g\in F[x,e^x,e^{-x}].$
	
	(\ref{antitower-inftint}):  It is easily seen using "integration by parts" that elements of the form $f\log^n(x)$ are $\infty-$integrable for any $n\in \mathbb N$ and $f\in C[x, 1/x].$ To prove the converse, we first apply Proposition \ref{inftinteg-simpleextn} \eqref{firstordercase} and obtain that if $g\in E$ is $\infty-$integrable, then there exist $n\in \mathbb{Z}_{\geq 0}$, $f_1,\dots, f_n \in C(x)$ with $f_n\neq 0$ such that  
    $$g= f_n\log^n(x)+f_{n-1}\log^{n-1}(x)+\cdots+f_0.$$
     Now we shall use an induction on $n$ to prove that $g\in C[x,1/x][\log(x)].$

	Noting that elements of $C(x)$, which are $\infty-$integrable in $E$, are precisely the elements of $C[x, 1/x],$ we assume that any $h\in C(x)[\log(x)]$ that is $\infty-$integrable and having $\deg_{\log(x)}(h)\leq n-1$  belongs to $C[x,1/x][\log(x)].$  Now,
	for any nontrivial differential automorphism $\sigma\in \mathscr G(E|C(x))$ with 
    $$\sigma(\log(x))=\log(x)+c_{\sigma}$$ 
    for some $0 \neq c_{\sigma} \in C$. We observe that $\sigma(g)$ is also $\infty-$integrable and therefore 
    $$\sigma(g)-g=nc_{\sigma}f_n \log^{n-1}(x)+\cdots $$ 
    is $\infty-$integrable, which  is a polynomial in $\log(x)$ over $C(x)$ of degree $n-1$ with leading coefficient $nc_\sigma f_n \in C(x).$  
    Thus, by the induction assumption we have $f_n\in C[x, 1/x].$
    Since $$g-f_n\log^n(x)=f_{n-1}\log^{n-1}(x)+\cdots+f_0$$ is also $\infty-$integrable and of degree $n-1$, the induction assumption implies that $f_{n-1}, \dots, f_0\in C[x,1/x]$ as desired.

(\ref{nthrootextn}): Note that $f\in C[x]$ is $\infty-$integrable in $C[x]$ and that for any $i \in \mathbb{Z}$ the element $x^{i/n}$ is also $\infty-$integrable in $E.$ For any $f\in C[x,x^{-1}],$ we see that $f x^{i/n}$ is a sum of terms of the form $cx^{m/n}$ with some $m\in \mathbb Z.$ Thus, for each $1\leq i\leq n$ the element $fx^{i/n}$ is $\infty-$integrable and consequently, $g=\sum_{i=0}^{n-1} f_i x^{i/n}$ is also $\infty-$integrable.

Conversely, let $g \in E$ be $\infty-$integrable in $E$.  Then, for each $j\in \mathbb N$ there is an element $g_j\in E$ such that $g^{(j)}_j=g.$ We write $g = \sum_{i=0}^{n-1} f_i x^{i/n}$ and $g_j=\sum_{i=0}^{n-1} f_{i,j}x^{i/n}$ with $f_i\in C(x)$ and $f_{i,j}\in C(x)$.
Computing $g^{(j)}_j$ and comparing the result with $g$ we obtain for all $0\leq i \leq n-1$ that    
$$f_{i}x^{i/n}=\left(f_{ij}x^{i/n}\right)^{(j)}=\left(f^{(j)}_{i,j}+\frac{ij}{nx}f^{(j-1)}_{i,j}+\dots\right)x^{i/n}$$ and so
\begin{equation}\label{polecomparison-algextn}
f_i=f^{(j)}_{i,j}+\frac{ij}{nx}f^{(j-1)}_{i,j}+\dots.
\end{equation}
For $i=0$ we then have $f_0=f_{0,j}^{(j)}$ for all $j\in \mathbb N$, meaning that $f_0$ is $\infty$-integrable, and so $f_0\in C[x].$ We shall now show that for all $1\leq i\leq  n-1$ the element $f_i$ has no pole at any $ 0\neq c \in C$, that is, $\mathrm{ord}_c(f_i)\geq 0$. This would then prove that each $f_i\in C[x,x^{-1}].$
Suppose on the contrary that for some $0\neq c\in C$ and some $1\leq i \leq n-1$ we have $\mathrm{ord}_c(f_i)=r<0.$  Then, from Equation \eqref{polecomparison-algextn}, we have that $\mathrm{ord}_c\left(f_{i,j}\right)<0$  for each $j\in \mathbb N$ and so 
$$
\mathrm{ord}_c\left(f^{(j)}_{i,j}+\frac{i,j}{nx}f^{(j-1)}_{i,j}+\dots\right)=\mathrm{ord}_c\left(f^{(j)}_{i,j}\right)
\ \mbox{for each} \  j\in \mathbb{N} \, .
$$  
In particular, for $j=|r|$ we get  $$r=\mathrm{ord}_c(f_i)=\mathrm{ord}_c\left(f_{i,|r|}\right)+r< r,$$ which is absurd. This proves (\ref{nthrootextn}).
\end{proof}


	\bibliographystyle{alpha}
	\bibliography{SSS}

\end{document}